\documentclass[11pt]{article}
\usepackage{amsmath}
\usepackage{amsfonts}
\usepackage{latexsym}
\usepackage{amssymb}
\usepackage{graphicx,psfrag}
\usepackage{amsthm}
\usepackage{fullpage}
\usepackage{framed}
\usepackage{verbatim}
\usepackage{color}
\usepackage{epsfig}
\usepackage{authblk}
\usepackage{hyperref}
\usepackage{geometry}
\usepackage{mathtools}
\usepackage{float}
\usepackage{enumerate}
\theoremstyle{plain}
\newtheorem{theorem}{Theorem}[section]
\newtheorem{corollary}[theorem]{Corollary}

\newtheorem{lemma}[theorem]{Lemma}

\theoremstyle{definition}
\newtheorem{definition}[theorem]{Definition}
\newtheorem{example}[theorem]{Example}
\newtheorem{notation}[theorem]{Notation}

\newtheorem{remark}[theorem]{Remark}

\theoremstyle{remark}

\newcommand\ZZ{\mathbb{Z}}
\newcommand\op{\operatorname}

\newcommand\R{\mathbb{R}}
\newcommand\Rplus{\mathbb{R}_{>0}}

\def\SS{S}
\def\CC{\mathcal C}
\def\RR{\mathcal R}
\newcommand{\supp}{\op{supp}}

\newcommand{\ra}{\rightarrow}

\newcommand{\rpath}{\rightarrow^*}
\newcommand{\notrpath}{\not\rpath}
\newcommand{\dra}{\ra}

\newenvironment{bullets}%
        {\begin{list}
                {\noindent\makebox[0mm][r]{$\bullet$}}
                {\leftmargin=5.5ex \usecounter{enumi}
      \topsep=1.5mm \itemsep=-.75ex}
        }
        {\end{list}}
\begin{document}
\title{Autocatalysis in Reaction Networks.}
\author[1]{Abhishek Deshpande\thanks{deshabhi123@gmail.com}}
\author[2]{Manoj Gopalkrishnan\thanks{manojg@tifr.res.in}}
\affil[1]{Center for Computational Natural Sciences and Bioinformatics, International Institute of Information Technology, Gachibowli, Hyderabad, India.}
\affil[2]{School of Technology and Computer Science\\ Tata Institute of Fundamental Research, Mumbai, India.}
\date{6 August 2014}
\maketitle
\begin{abstract}
The persistence conjecture is a long-standing open problem in chemical reaction network theory. It concerns the behavior of solutions to coupled ODE systems that arise from applying mass-action kinetics to a network of chemical reactions. The idea is that if all reactions are reversible in a weak sense, then no species can go extinct. A notion that has been found useful in thinking about persistence is that of ``critical siphon.'' We explore the combinatorics of critical siphons, with a view towards the persistence conjecture. We introduce the notions of ``drainable'' and ``self-replicable'' (or autocatalytic) siphons. We show that:	 every minimal critical siphon is either drainable or self-replicable; reaction networks without drainable siphons are persistent; and non-autocatalytic weakly-reversible networks are persistent. Our results clarify that the difficulties in proving the persistence conjecture are essentially due to competition between drainable and self-replicable siphons.
\end{abstract}

\section{Introduction}
Following \cite[Example~6.3]{gopalkrishnan}, consider the reaction network in the species $X$ and $Y$ with the reactions
\begin{align*}
Y & \xrightleftharpoons[k_2]{\,k_1\,} 2X, &&2Y \xrightleftharpoons[k_4]{\,k_3\,} X
\end{align*}
Let $x(t),y(t)$ represent concentrations at time $t$ of the species $X,Y$ respectively. Then \textbf{mass-action kinetics} is described by the system of ordinary differential equations
\begin{align*}
\left(\begin{array}{c}
\dot{x}\
\\\dot{y}\end{array}\right) =
\left(\begin{array}{c}
2\
\\-1\end{array}\right) k_1 y(t) +  \left(\begin{array}{c}
-2\
\\1\end{array}\right)k_2x(t)^2 + \left(\begin{array}{c}
1\
\\-2\end{array}\right)k_3 y(t)^2 +\left(\begin{array}{c}
-1\
\\2\end{array}\right) k_4 x(t)
\end{align*}
Does there exists a trajectory $(x(t),y(t))$ originating in the positive orthant (i.e., $(x(0),y(0))\in \R^2_{>0}$) that can get arbitrarily close to the boundary $\partial \R^2_{\geq 0}=\R^2_{\geq 0}\setminus \R^2_{>0}$ as time $t\to\infty$? That is, for each $\epsilon > 0$, does there exist a time $t_\epsilon$ such that the distance of $(x(t_\epsilon), y(t_\epsilon))$ from $\partial \R^2_{\geq 0}$ is less than $\epsilon$? The \textbf{persistence conjecture}~\cite{Feinberg89} asserts that so long as the reaction network is \emph{weakly-reversible}~(Definition~\ref{def:rnprop}), this can not happen. This conjecture has a history going back to 1974~\cite{horn74dynamics}. It remains open, even when all reactions are reversible, \emph{even for as few as $3$ species}! The $2$ species case was settled only very recently~\cite{CNP,Pantea}.

The notion of \textbf{critical siphon}~\cite{Angeli} provides a powerful and easy combinatorial way of proving persistence for a certain subclass of networks. The idea is that the existence of a trajectory violating persistence has certain combinatorial implications, i.e., the existence of critical siphons. Conversely, the absence of critical siphons --- which can be combinatorially checked --- has dynamical implications, i.e., persistence.

For example, consider the following network in the species $X$ and $Y$ with reactions given by
\begin{align*}
0\rightleftharpoons X + Y \rightleftharpoons 2X + Y \rightleftharpoons X + 2Y \rightleftharpoons 3X + 4Y
\end{align*}
This is persistent. The idea is that if $X$ were ever at zero concentration, it would be produced from nothing, and similarly for $Y$. Only the first reaction $0\rightleftharpoons X+Y$ matters for this analysis. The other reactions can be replaced arbitrarily, and so long as they remain reversible, we can prove persistence. The notion of \emph{siphon} helps to make this idea into a proof. A siphon is a set of species that, once absent, remain absent. It turns out that reaction networks without siphons are persistent.

Another example is $X\rightleftharpoons Y$. Here if $X$ and $Y$ are both absent, they remain absent, so $\{X,Y\}$ is a siphon. However, there is a conservation law: $x(t) + y(t)$ is time-invariant, because whenever an $X$ is destroyed, a $Y$ is created, and vice versa. Therefore, if the initial conditions are $(x(0),y(0))\in\R^2_{>0}$ then the point $(0,0)$ where both $X$ and $Y$ are absent is unreachable. The idea is that if there is a positive conservation law whose support is contained in a siphon then that siphon can not be drained. Hence, among siphons, the challenge to persistence comes only from \emph{critical} siphons, i.e., siphons that do not contain the support of a positive conservation law. The example with reactions $Y\rightleftharpoons 2X$ and $2Y\rightleftharpoons X$ has the critical siphon $\{X,Y\}$. So, though it is persistent, showing this requires a different argument that proceeds by exhibiting a Lyapunov function~\cite{CNP,geometricgac}.

A different approach to the persistence conjecture is to try to identify a subclass of weakly-reversible systems which are somehow \emph{physical}, and exploit this extra structure to obtain a proof for the persistence conjecture. This program was initiated in \cite{adleman-2008} and \cite{gnacadja}, where the notion of systems where every species is made up of \emph{atoms} in a unique way was made precise, and it was shown that atomic systems have no critical siphons, and hence are persistent. Later, in \cite{gopalkrishnan}, this result was extended to all \emph{saturated} or non-catalytic systems, a larger class of systems properly containing the atomic systems.

The biological motivation for our work comes from Synthetic Biology, specifically from the DNA Molecular Programming community~\cite{SeesawQianWinfree, soloveichik2010dna,phillips2009programming}, where reaction networks are being viewed as programming languages for the synthesis of desired behaviour. One would like to have a programming language that is natural with respect to chemical reaction dynamics, so that there are meaningful ways to separately test, and then compose, different subsystems. The Petri net view of reaction networks makes a connection with process algebras, and connects to work on concurrent programming languages. Our work may be viewed as setting down a mathematical bridge between these two different areas, by making explicit the combinatorial foundations of reaction network notions like catalysis, self-replication, and persistence. We hope to exploit these connections more in future to describe ways of programming chemical reaction networks that will take advantage of advances in concurrent programming language design.

Our contributions in this paper are the following:
\begin{bullets}

\item We give an explicit combinatorial characterization of the minimal critical siphons. Specifically, we introduce the notions of \textbf{drainable} and \textbf{self-replicable} sets in Definition~\ref{def:siphon}, and prove in Theorem~\ref{minimal-siphon} that every minimal critical siphon is either drainable or self-replicable.

\item Conjecture~1 from \cite{gopalkrishnan} asserted a link between a certain notion of autocatalysis and critical siphons. We provide a counterexample~(Example~\ref{ex:counter}) to this conjecture. Our notion of self-replicable siphons captures a more nuanced notion of autocatalysis. With this notion, we show in the spirit of \cite[Conjecture~1]{gopalkrishnan} that the weakly-reversible networks that have critical siphons are precisely the autocatalytic systems~(Theorem~\ref{minimal-siphon}). In particular, all non-autocatalytic (i.e., without self-replicable siphons) weakly-reversible networks are persistent. As a special case, in Section~\ref{sec:catalysis} we obtain a combinatorial proof for the persistence of non-catalytic networks, the main result in \cite{gopalkrishnan} which was originally proved using algebraic geometric methods.

\item Angeli et al.~\cite{Angeli} have shown that conservative reaction networks without critical siphons are persistent. We show that reaction networks without drainable siphons are persistent~(Theorem~\ref{thm:persistence}). Since networks with drainable siphons are a proper subclass of networks with critical siphons~(Theorem~\ref{minimal-siphon}), our result sharpens Angeli's result, and brings out that the obstruction to proving the persistence conjecture comes from a competition between the drainable and self-replicable natures of a siphon~(Remark~\ref{Rmk:compete}).

\item In Section~\ref{sec:dm}, towards proving Theorem~\ref{minimal-siphon}, we prove the Convex Rank-Nullity theorem (Theorem~\ref{rank-nullity}), a trichotomy result for matrices that have a sign pattern with all diagonal terms are negative, and all off-diagonal terms positive. Though this result is very suggestive of Farkas' lemma and its variants, to our best knowledge it has neither appeared before, nor is it an immediate corollary of Farkas' lemma~(Remark~\ref{rmk:notfarkas}).

\item All our results are in the setting of \emph{positive} reaction networks (Definition~\ref{def:rn}). This means our results hold for reactions like $0.3 X + 2.14 Y \to 1.1 Z$ with fractional stoichiometric coefficients.

\item We introduce the notion of a \emph{critical} set (Definition~\ref{def:siphon}), so that a critical siphon is a set of species which is both critical and a siphon. Together with the notion of \emph{critical point}~(Definition~\ref{def:critpt}), this decoupling of the two notions simplifies the analysis~(Theorem~\ref{critical-equivalence}), and allows us to obtain slightly stronger results~(Theorem~\ref{minimal-siphon}).
\end{bullets}

Figure~\ref{fig:1} provides a summary of all the results at a glance.

\begin{figure}[ht!]
\caption{\label{fig:1}Summary of results.}
\begin{center}
\[
\includegraphics[scale=.65]{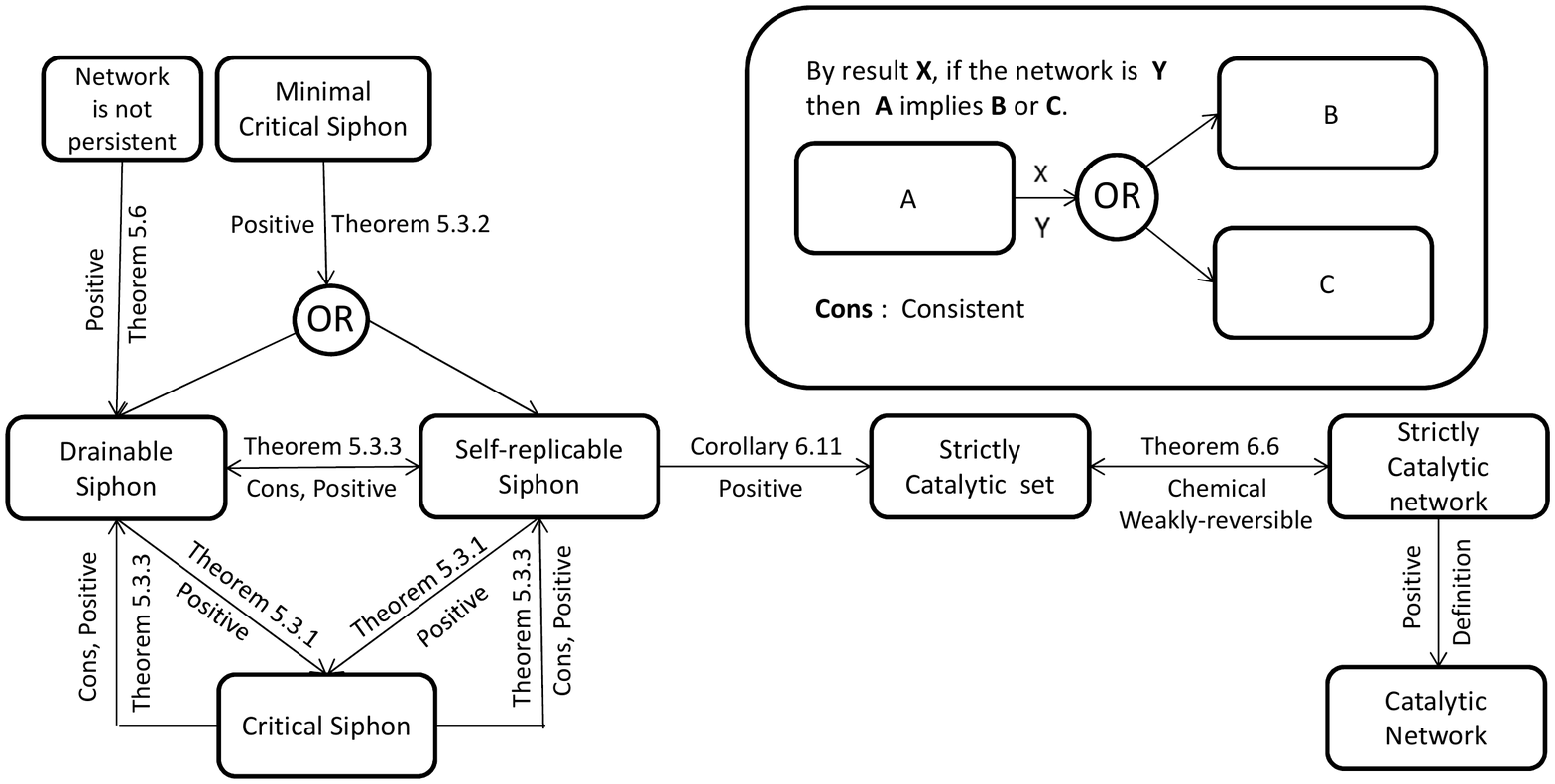}
\]
\end{center}
\end{figure}

\section{Preliminaries}
\begin{notation}
Let $z,z'\in\R^n$ be two vectors.
\begin{enumerate}
\item $\min(z, z')$ is the vector whose $i$'th coordinate is $\min(z_i,z'_i)$ for all $i\in\{1,2,\dots,n\}$.
\item $\lceil z \rceil$ is the vector whose $i$'th coordinate is $\lceil z_i\rceil$ for all $i\in\{1,2,\dots,n\}$.
\item $\supp(z):= \{ i \mid z_i \neq 0\}$.
\end{enumerate}
\end{notation}

\begin{definition}\label{def:cone}
For all vectors ${v_1,v_2,\dots, v_k\in\R^n}$:
\begin{enumerate}
\item The {\em convex polyhedral cone}
$\op{Cone}\{v_1,v_2,\dots,v_k\}:= \left\{\displaystyle\sum_{i=1}^k a_i v_i \mid a_1,a_2,\dots,a_k \geq 0\right\}$.
\item $\op{Span}\{v_1,v_2,\dots,v_k\}:= \left\{\displaystyle\sum_{i=1}^k a_i v_i \mid a_1,a_2,\dots,a_k \in\R\right\}$.
\end{enumerate}
\end{definition}

\begin{definition}[Reaction Network]\label{def:rn}
For every finite set $\SS$:
\begin{enumerate}
\item An $\SS$-{\em complex} is a vector $y\in \R^\SS$.
\item A complex $y\in\R^\SS$ is
  \begin{enumerate}
  \item {\em Positive} iff $y\in\R^\SS_{\geq 0}$.
  \item {\em Integer} iff $y\in\mathbb{Z}^\SS$.
  \item {\em Chemical} iff $y\in\mathbb{N}^\SS$.
  \end{enumerate}
\item A {\em reaction} is a pair $(y,y')$ of complexes, written $y\ra y'$ with {\em reactant} $y$ and {\em product} $y'$.
\item A reaction $y\ra y'$ is {\em positive} (respectively, {\em integer, chemical}) iff both $y,y'$ are positive (respectively, integer, chemical).
\end{enumerate}
A {\em reaction network} is a pair $(\SS,\RR)$ where $\SS$ is a finite set denoting reacting species, and $\RR$ is a finite set of reactions. The {\em reaction graph} of the reaction network $(\SS,\RR)$ is the directed graph with vertex set $\CC:=\{y,y'\mid y\to y' \in\RR\}$ and edge set $\RR$.
\end{definition}

\begin{remark}
Conventionally, a reaction network is defined as a triple $(\SS,\CC,\RR)$ where $\CC$ is the set of complexes that occur either as reactants or as products in some reaction in $\RR$. Explicitly providing the set $\CC$ is redundant, because the complexes can be recovered from the set of reactions. Therefore, we have found it more economical to depart from convention, and define a reaction network as a two-tuple of species and reactions.
\end{remark}

\begin{notation}
Let $G=(\SS,\RR)$ be a reaction network. A reaction $y\ra y'$ is a {\em $G$-reaction} iff $y\ra y'\in\RR$.
\end{notation}

\begin{definition}
Let $G=(\SS,\RR)$ be a reaction network. Then
\begin{enumerate}
\item The {\em stoichiometric subspace} $H\subseteq\R^\SS$ of $G$ is the linear span of the reaction vectors $\{ y' - y \mid y\ra y'\in\RR\}$.
\item A $G$-{\em conservation law} is a linear polynomial $\sum_{i\in\SS} w_i x_i = w\cdot x$, where the $x_i$'s are indeterminates, and $w\in\R^\SS$ is perpendicular to $H$.
\item A $G$-conservation law $w.x$ is {\em positive} iff $w\in\R^\SS_{\geq 0}\setminus\{0\}$.
\end{enumerate}
\end{definition}

\begin{remark}\label{rem:conslaw}
It is well-known and easily verified that $G$-conservation laws are invariant to the application of $G$-reactions.
That is, if $y\rightarrow y'$ is a $G$-reaction and $\sum_i w_i x_i$ is a $G$-conservation law  then $\sum_{i\in\SS} w_i x_i = \sum_{i\in\SS}w_i (x_i + \theta (y'_i - y_i))$ for all $\theta\in\R$. Indeed, this is but a special case of the observation that $\sum_{i\in\SS} w_i x_i = \sum_{i\in\SS}w_i (x_i + \theta \,v_i)$for all vectors $v$ in the stoichiometric subspace $H$ of $G$.
\end{remark}
\begin{definition}
Let $G=(\SS,\RR)$ be a reaction network. Then the {\em stoichiometric matrix} $\Gamma$  of $G$ is the matrix whose rows are the reaction vectors $\{ y' - y \mid y\ra y'\in\RR\}$. For all $T\subseteq\SS$, the {\em stoichiometric submatrix $\Gamma_T$} is the matrix obtained from $\Gamma$ by keeping only those columns that correspond to species in $T$.
\end{definition}

\begin{definition}\label{def:rnprop}
A reaction network $G$ is:
\begin{enumerate}
\item {\em Positive}\label{def:pos} (respectively {\em integer, chemical}) if each $G$-reaction is positive (respectively integer, chemical).
\item {\em Reversible} if the reaction graph is undirected.
\item {\em Weakly-reversible} if in the reaction graph, every reaction belongs to a cycle.
\item {\em Consistent}\label{cons} if there exists a strictly positive vector $v$ such that $v\cdot\Gamma = 0$
\end{enumerate}
\end{definition}

\begin{remark}\label{wr-cons}
The notion of \emph{consistent} is well-known from the Petri net literature~\cite{petri}. It is also well-known~\cite{Angeli} that every weakly-reversible reaction network is consistent. For a reaction network consisting of a single cycle, this is obvious by taking $v$ to be the vector of all $1$'s. Since a weakly-reversible network can be expressed as a combination of cycles, the result follows for an arbitrary weakly-reversible network.
\end{remark}

\begin{example}
$2X + 3Y$ is a complex, represented by the vector $(2,3)$ in the space $\R^{\{X,Y\}}\cong\R^2$ of complexes. It is a positive complex because $2,3\geq 0$. It is an integer complex because $2,3$ are integers. Finally, it is a chemical complex because $2,3$ are positive integers.

The reactions $X \ra 2Y$ and $2X\ra Y$ are chemical reactions, where the reaction notation is understood as shorthand to represent the pairs of vectors $( (1,0),(0,2))$ and $( (2,0), (0,1))$. The corresponding reaction network is $G=(\{X,Y\},\{X\ra 2Y, 2X\ra Y \})$.
\end{example}

\begin{definition}
Let $G=(\SS,\RR)$ be a reaction network. For all reactions $y\ra y'$ and positive complexes $w\in\R^\SS_{\geq 0}$, the {\em dilution} of $y\ra y'$ by $w$ is the reaction $y+w\ra y'+w$. A reaction $z\ra z'$ is a $G$-{\em dilution} iff there exist a positive complex $w\in\R^\SS_{\geq 0}$ and a $G$-reaction $y\ra y'$ such that $z\ra z'$ is the dilution of $y\ra y'$ by $w$.
\end{definition}

Dilutions of reactions have been defined by Cardelli~\cite{cardelliStrand2011}.

\begin{notation}
Let $G=(\SS,\RR)$ be a reaction network, and let $y,y'\in\R^\SS$ be complexes. One says that there is a $G$-{\em reaction pathway} $y\rpath_G y'$ iff there exists a sequence ${y=y(0)\dra y(1)\dra\dots \dra y(k)=y'}$ of $G$-dilutions.
\end{notation}

Clearly, $G$-dilutions need not be $G$-reactions. Such dilutions describe the effect of applying a $G$-reaction $y\ra y'$ to a population $y+w$ that contains all the reactants required for the reaction to fire, to obtain the population $y' + w$.

A $G$-reaction pathway corresponds to the idea of reachability: given a population $z$ in a test-tube, and a reaction network, what other populations are reachable by applying $G$-reactions? The reachable populations are precisely the populations $z'$ such that $z\rpath_G z'$. This notion of reachability has been considered by other authors~\cite{soloveichikComputation2008, Chen2012,cardelliStrand2011}.

\begin{lemma}\label{lem:reactionpathway}
Let $G$ be a reaction network, and let $z\rpath_G z'$ be a $G$-reaction pathway. Then $z'- z$ is in the positive span $H_+ = \op{cone}\{y'- y \mid y\ra y'\in\RR\}$ of the reactions of $G$. In particular, for every $\theta\in\R$, for every $G$-reaction pathway $y\rpath_G y'$, we have $w\cdot (x + \theta(y'-y)) = w\cdot x$ for every $w$ perpendicular to the stoichiometric subspace $H$.
\end{lemma}
\begin{proof}
To see that $z'- z \in H_+$, first note that for $G$-dilutions $x\ra x'$ we have $x'-x\in H$. Since reaction pathways are formed by a chain $z = y(0)\ra y(1)\ra y(2) \dots \ra y(k) = z'$ of dilutions, we have $z'- z = y(k) - y(0) = y(k) - y(k-1) + y(k-1) - y(k-2) + \dots + y(1) - y(0)$, which is in $H_+$. The invariance of $w\cdot x$ along $z\rpath_G z'$ now follows from Remark~\ref{rem:conslaw}.
\end{proof}

\begin{example}\label{ex:x2y}
Let $G=(\{X,Y\},\{X\ra 2Y, 2Y\ra X, 2X\ra Y, Y\ra 2X \})$. Then the reaction $2X \ra X + 2Y$ is not a $G$-reaction. However, it is a dilution by $X$ of the $G$-reaction $X \ra 2Y$. The sequence of $G$-dilutions
\[
2X \ra X + 2Y \ra 3X + Y \ra 5X
\]
where the second and third reactions are dilutions of $Y\ra 2X$, yields a reaction pathway $2X \rpath_G 5X$.
\end{example}

\begin{definition}\label{def:opp_rn}
Let $G=(\SS,\RR)$ be a reaction network. The {\em opposite reaction network} $G^{op}$ of $G$ is the reaction network with species $\SS$ and reactions $\{ y' \ra y \mid y \ra y' \in \RR \}$ obtained by reversing the reactions of $G$.
\end{definition}

\section{Siphons, Drainability, Self-replicability}\label{sec:sdsr}
We will say that a set of species is \textbf{self-replicable} if there is an initial state, and a reaction pathway from that state that will increase the numbers of each species in the set. In ecological terms, this would capture the potential for a set of species to simultaneously prosper and increase in number. The opposite notion is that of \textbf{drainable} sets of species, where from an initial state a reaction pathway can deplete the numbers of each species. In ecological terms, this notion captures the potential for a set of species to simultaneously decrease in number, and is connected to the threat of extinction or non-persistence. We now define these notions more formally.

\begin{definition}\label{def:siphon}
Let $G=(\SS,\RR)$ be a reaction network. A set $T\subseteq\SS$ of species is
\begin{enumerate}

\item  A {\em siphon} iff for every reaction $y \rightarrow y'\in\RR$, if the complex $y'$ contains a species from $T$ then the complex $y$ contains a species from $T$.

\item\label{def:critical} {\em Critical} iff there is no positive conservation law whose support is contained in $T$.

\item\label{def:sr} {\em Self-replicable} iff there is a reaction pathway $y\rpath_G y'$ such that for every species $i\in T$, the $i$'th component $(y'-y)_i$ is strictly positive.

\item\label{def:dr} {\em Drainable} iff $T$ is self-replicable for $G^\text{op}$, i.e., there is a reaction pathway $y\rpath_G y'$ such that for every species $i\in T$, the $i$'th component $(y'-y)_i$ is strictly negative.

\item {\em Closed} iff for every reaction $y\ra y'\in\RR$, if $\supp{y}\subseteq T$ then $\supp{y'}\subseteq T$. The {\em closure} $\op{Cl}(T)$ is the smallest closed set containing $T$.

\end{enumerate}
\end{definition}

Remark~\ref{Rmk:1}.\ref{intersection} will show that the intersection of closed sets is closed, and hence the closure is unique.

\begin{notation}
As is usual, one calls a set {\em maximal} (respectively, {\em minimal}) among all nonempty sets having a property $P$ if no set having property $P$ properly contains it (is properly contained in it). In this sense, we shall refer to maximal and minimal siphons, critical siphons, self-replicable sets, and drainable sets. In particular, note that a minimal critical siphon is minimal among all critical siphons, but not necessarily minimal among all critical sets. Also see Remark~\ref{Rmk:1}.\ref{rmk:mincritsiph}.
\end{notation}

Siphons and critical siphons have been defined before in the chemical reaction network literature~\cite{Angeli}. Though ``siphon'' appears to be the more conventional term, such sets have also been called semi-locking sets. We find this to be a better description of their nature. They don't necessarily siphon off any flow; however, if all the siphon species are absent, then they remain absent. Critical siphons have also been referred to as ``relevant siphons.''~\cite{shiu}

The notion of ``closure'' with respect to a set of reactions of a starting ``food'' set of species  has been considered before in the literature of autocatalytic sets~\cite{hordijk_steel}. Gnacadja~\cite{gnacadja} has considered the notions of ``reach-closure'' of a set. Our notion of closed sets is cognate with both these notions.

\begin{example}
Consider the reaction network $G$ with reactions\[X\leftrightharpoons 2Y, \mbox{  }2X\leftrightharpoons Y \] from Example~\ref{ex:x2y}. The set $\{X,Y\}$ is:
\begin{enumerate}

\item A siphon, because every reaction with either $X$ or $Y$ on the product side requires either $X$ or $Y$ on the reactant side. The set $\{X\}$ is not a siphon, because of the $G$-reaction $Y\ra 2X$.

\item Critical because there are no conservation laws.

\item Self-replicable because there is a $G$-reaction pathway
\[{2X + 2Y \rpath_G 3X + 3Y} = 2X + 2Y \to X + 4Y \to 3X + 3Y\]
increasing both the number of $X$ and number of $Y$.

\item Drainable because there is a $G$-reaction pathway
\[
{3X + 3Y \rpath_G 2X + 2Y} = 3X + 3Y\to X + 4Y \to 2X + 2Y
\]
decreasing both the number of $X$ and the number of $Y$.

\item Closed. In particular, the set of all species is trivially always closed. The set $\{X\}$ is not closed, because the $G$-reaction $X\ra 2Y$ produces the species $Y$ which was not in the set.

\end{enumerate}
\end{example}

The next remark collects some easy observations about these definitions.

\begin{remark}\label{Rmk:1}
For every reaction network:
\begin{enumerate}
\item\label{closedsiphon} As shown in \cite[Proposition~3.4]{gnacadja}, a set $T$ is a siphon iff its complement $\SS\setminus T$ is closed. To see this, observe that\
\\$\SS\setminus T$ is not closed\
\\$\Leftrightarrow$ There exists a reaction $y\rightarrow y'$ such that the complex $y$ does not contain any species from $T$, while the complex $y'$ contains a species from $T$.\
\\$\Leftrightarrow$ $T$ is not a siphon.

\item\label{intersection} The union of siphons is a siphon. The intersection of closed sets is closed. In particular, this is why the definition of closure of a set is well-defined.

\item The intersection of siphons need not be a siphon. Consider the following reactions: \{$X + Y\rightarrow Y$, $X+Z \rightarrow Y$\}. In this example, the sets $\{X,Y\}$ and $\{Y,Z\}$ are both siphons, but their intersection $\{Y\}$ is not a siphon. For the same reason, the union of closed sets need not be closed.

\item \label{subset-critical}Every subset of a critical set is critical. If not, suppose $T'$ is a non-critical subset of $T$. Then there exists a positive conservation law $w.x$ with $\supp{w}\subseteq T'\subseteq T$, contradicting $T$ being critical.

\item\label{rmk:mincritsiph} As a consequence, every minimal critical siphon is also minimal among siphons.

\item \label{subset-self} Every subset $P$ of a self-replicable set $T$ is self-replicable since if $y\rpath_G y'$ is such that the $i$'th component $y'_i - y_i > 0$ for all $i\in T$ then a fortiori $y'_i - y_i > 0$ for all $i\in P$.

\item \label{subset-drainable}Similarly, every subset of a drainable set is drainable.

\item The union of critical sets need not be critical. Consider the reaction $X \rightarrow Y$. In this reaction the sets $\{X\}$ and $\{Y\}$ are both critical, but their union $\{X,Y\}$ is not critical because of the positive conservation law $x + y$.

\item\label{closedunderunion} The union of self-replicable sets/ drainable sets need not be self-replicable/ drainable. Consider the following reactions : \{$2X \rightarrow Y$, $2Y \rightarrow X$\}. In this example, the sets $\{X\}$ and $\{Y\}$ are self-replicable, but their union $\{X,Y\}$ is not self-replicable. Reversing all reaction arrows gives an example where the union of drainable sets is not drainable.

\end{enumerate}
\end{remark}

\begin{lemma}\label{lem:siphonrpath}
Let $G=(\SS,\RR)$ be a reaction network. A set $T\subseteq\SS$ of species is a siphon iff for every reaction pathway
$y \rpath_G y'$, if the complex $y'$ contains a species from $T$ then the complex $y$ contains a species from $T$.
\end{lemma}
\begin{proof}
($\Leftarrow$:) Since every reaction is a reaction pathway.\
\\($\Rightarrow$:) Suppose $T$ is a siphon. Then by definition, for every reaction $y\ra y'\in\RR$, if the complex $y'$ contains a species from $T$ then the complex $y$ contains a species from $T$. In particular, this is true for every dilution $y+w\to y'+w$ of $y\to y'$, where $w\in \R^\SS_{\geq 0}$. Therefore, it must be true for chains of dilutions, i.e., reaction pathways.
\end{proof}

\begin{definition}[Critical point]\label{def:critpt}
Let $G=(\SS,\RR)$ be a reaction network with stoichiometric subspace $H = \op{span}\{y'-y\mid y\rightarrow y'\in\RR\}$. A point $z\in\R^S_{\geq 0}$ is a {\em critical point} iff the affine space $z + H$ meets the positive orthant $\R^\SS_{>0}$.
\end{definition}

Note that a critical point is a point whose support is {\em stoichiometrically admissible} in the language of \cite{gnacadja}. Critical points are a generalization of critical siphon-points defined in \cite{gopalkrishnan}. Firstly, we do not require the reaction network to be weakly-reversible. Further, we do not require the point $z$ to be an equilibrium point, nor do we require the set $\{ i \mid z_i = 0\}$ to be a siphon.

\begin{theorem}\label{critical-equivalence}
Let $G=(\SS,\RR)$ be a reaction network. Let $T\subseteq\SS$. Then the following are equivalent:
\begin{enumerate}
\item $T$ is critical.
\item $\ker{\Gamma_T}\cap \R^T_{\geq 0} = \{0\}$.
\item Every point $z\in\R^\SS_{\geq 0}$ with $\{i\mid z_i =0\}=T$ is a critical point.
\item There exists a critical point $z\in\R^\SS_{\geq 0}$ with $z_i =0$ iff $i\in T$.
\end{enumerate}
\end{theorem}

\begin{proof}
Let $z\in\R^\SS_{\geq 0}$ be such that $z_i =0$ iff $i\in T$. Let $H = \op{span}\{y' - y \mid y\ra y'\in\RR\}$ be the stoichiometric subspace of $G$.\
\
\\$\neg(2)\Rightarrow\neg(1)$. If $0\neq v\in\ker{\Gamma_T}\cap\R^T_{\geq 0}$ then $v\cdot x$ is a positive conservation law with support in $T$, and $T$ is not critical.\
\
\\$(2)\Rightarrow(3)$. From the Transposition theorem of Gordan~\cite[(31), pp.~95]{schrijver} applied to $\Gamma_T$, there exists $v\in\R^\RR$ such that $v \Gamma_T>0$. Note that $v \Gamma$ is in the stoichiometric subspace $H$ of $G$. Consider $w := z + \epsilon v \Gamma$ where $\epsilon > 0$ is to be chosen sufficiently small. It is enough to show that there exists $\epsilon > 0$ so that $w \in\R^\SS_{>0}$. We show this by case analysis.
\\Case 1: $i\in T$. Since $z_i=0$, we have $w_i = \epsilon (v \Gamma)_i > 0$.\
\\Case 2: $i\notin T$. Then $z_i > 0$ and so for  $\epsilon < \frac{z_i}{ |(v\Gamma)_i|}$, we have $w_i = z_i + \epsilon (v\Gamma)_i > 0$.\
\
\\$(3)\Rightarrow(4)$ is obvious.\
\
\\$\neg(1)\Rightarrow\neg(4)$. Suppose $T$ is not critical. Then there exists a positive conservation law $w\cdot x$ such that $\supp{w}\subseteq T$. In particular, $w\cdot (z+H) = w\cdot z + w\cdot H = w\cdot z = 0$. But $w\cdot x\neq 0$ for all $x\in\R^\SS_{>0}$. We conclude that $z+H$ does not meet $\R^\SS_{>0}$. Hence, $z$ is not a critical point.
\end{proof}

\begin{example}\label{ex:counter}
Consider the reaction network
\[X\rightleftharpoons 2X, \mbox Y\rightleftharpoons X+Y.
\]
The associated {\em event-system}~\cite{adleman-2008} is $\mathcal{E}_G = \{x-x^2 , y - xy\}$. The network is autocatalytic as per \cite[Definition~6.2]{gopalkrishnan} because $x-x^2\in\mathcal{E}_G$ and for all $k\in \mathbb{N}$ the binomials $1-x^k$ are not contained in the ideal $(\mathcal{E}_G)$. However, contrary to \cite[Conjecture~1]{gopalkrishnan}, there are no critical siphons. The set $\{X\}$ is not a siphon because $X$ is produced by the second reaction. Neither of the siphons $\{Y\}$ and $\{X,Y\}$ are critical siphons, since the concentration of $Y$ is dynamically invariant.
\end{example}

Example~\ref{ex:counter} motivated our definition of drainable and self-replicable sets. The key observation was that the failure of \cite[Conjecture~1]{gopalkrishnan} can be attributed to working with only one species at a time. If we work with sets of species instead then, as we show in Theorem~\ref{minimal-siphon}, an appropriately modified version of \cite[Conjecture~1]{gopalkrishnan} is true.

\section{Diffusive Matrices}\label{sec:dm}
\begin{definition}
A matrix $A = (a_{ij})_{m\times n}$ is {\em diffusive} iff all diagonal elements are strictly negative, and all off-diagonal elements are non-negative. That is, $a_{ii}<0$ and $a_{ij}\geq 0$ for all $i\in\{1,2,\dots,m\}$ and $j\in\{1,2,\dots,n\}$ with $i\neq j$. A diffusive matrix is {\em strongly diffusive} iff all off-diagonal elements are strictly positive.
\end{definition}

\begin{example}
Laplacian matrices (or their negatives, depending upon convention) are diffusive. If the graph is a clique then the Laplacian matrix is strongly diffusive. If the graph is strongly connected then some power of the Laplacian matrix is strongly diffusive.
\end{example}

\begin{notation}
Let $A$ be a matrix. Then $\op{rows}(A)$ is the set of row vectors of $A$.
\end{notation}

\begin{lemma}\label{positive_cone}
Let $A=(a_{ij})_{n\times n}$ be a square matrix with $a_{ij}>0$ for all $i\neq j$.\\ If ${\op{cone}(\op{rows}(A)) \cap \R^n_{\geq 0} \setminus \{0\}}$ is nonempty then $\op{cone}(\op{rows}(A)) \cap \R^n_{>0}$ is nonempty.
\end{lemma}
\begin{proof}
Suppose $\op{cone}(\op{rows}(A)) \cap \R^n_{\geq 0} \setminus \{0\} \neq \emptyset$. Let $v\in\op{cone}(\op{rows}(A))\cap\R^n_{\geq 0}\setminus\{0\}$, and let $i$ be such that $v_i > 0$. Define
\[
w(\epsilon) := v + \epsilon ( a_{i1}, a_{i2},\dots, a_{in}).
\]
Then for sufficiently small $\epsilon > 0$, we claim that $w(\epsilon) \in \R^n_{>0}$, which completes the proof because $w(\epsilon) \in \op{cone}(\op{rows}(A))$. To see this, note first that $v_i > 0 \Rightarrow v_i + \epsilon a_{ii} = w(\epsilon)_i > 0$ for sufficiently small $\epsilon > 0$. Next, for all $j\neq i$, we have $w(\epsilon)_j = v_j + \epsilon a_{ij} > 0$ since $a_{ij} > 0$ by assumption.
\end{proof}

\begin{lemma}\label{lem:GE1}
Let $A$ be a diffusive square $n\times n$ matrix, and let $- A_1$ be the matrix obtained after one step of Gaussian elimination applied to $- A$. Then
\begin{enumerate}
\item $\op{cone}(\op{rows}(A_1))\subseteq \op{cone}(\op{rows}(A))$, and
\item Either $A_1$ is diffusive or $\op{rows}(A_1)\cap \R^n_{\geq 0}$ is nonempty.
\end{enumerate}
\end{lemma}

\begin{proof}
(1) Recall that Gaussian elimination proceeds by either (a) exchanging rows, or
(b) multiplying a row by a scalar, or (c) adding a constant times one row to another row. We consider each of these cases separately.\
\
\\(a) Since $A$ is diffusive, all elements on the diagonal of $-A$ are strictly positive, in particular non-zero. So the algorithm will not exchange rows.\
\\(b) If a row is multiplied to make the pivot element equal $1$, then it is multiplied by the reciprocal of the diagonal element, which is a positive number, so $\op{cone}(\op{rows}(A_1))\subseteq \op{cone}(\op{rows}(A))$.\
\\(c) If $\op{row}_j \leftarrow \op{row}_j + c \cdot\op{row}_i$ for some rows $i,j$ then the algorithm will choose $c = -\frac{a_{ji}}{a_{ii}} > 0$ so that
\begin{align*}
\op{cone}(\op{rows}(A_1)) &= \op{cone}(\{\op{row}_1,\op{row}_2,\dots, \op{row}_{j-1},\op{row}_j + c\cdot\op{row}_i,\op{row}_{j+1},\dots,\op{row}_n \})\
\\&\subseteq \op{cone}\{\op{row}_1,\op{row}_2,\dots, \op{row}_n,\op{row}_j + c\cdot\op{row}_i \}\
\\&= \op{cone}\{\op{row}_1,\op{row}_2,\dots, \op{row}_n \}\qquad (\text{Since }c>0)\
\\&= \op{cone}(\op{rows}(A)).
\end{align*}\
\
\\(2) Suppose $A_1$ is not diffusive. Since scalar multiplication preserves diffusivity, $A_1$ must have been obtained from $A$ by an operation $\op{row}_j \leftarrow \op{row}_j + c \op{row}_i$ for some $c > 0$ and rows $i,j$. Since $A$ was diffusive and by the choice of $c$, this operation preserves non-negativity of off-diagonal terms. By non-diffusivity of $A_1$, the diagonal entry on the $j$'th row of $A_1$ must be non-negative, so that the $j$'th row of $A_1$ is in $\R^n_{\geq 0}$.
\end{proof}

\begin{remark}
The reason we are applying Gaussian elimination to $-A$ instead of to $A$ is because Gaussian elimination as usually described in textbooks sets the pivot elements to be positive, whereas diffusive matrices have negative elements on the diagonal. If one were to either change the convention in the Gaussian elimination algorithm, or reverse the sign convention for what is meant by a diffusive  matrix, we wouldn't need this negative sign.
\end{remark}

\begin{lemma}\label{lem:GE2}
Let $A$ be a strongly diffusive square matrix. Suppose Gaussian elimination is applied to $-A$ to obtain the sequence $-A = -A_0, -A_1,-A_2,\dots, -A_k$ of matrices, and suppose $A_1, \dots, A_{k-1}$ are diffusive and $A_k$ is not. If $\op{rows}(A_k)\cap \R^n_{\geq 0} = \{0\}$ then $A_k$ has the sign pattern
\[ \left( \begin{array}{cccc}
<0 & \geq 0 & \dots      &\geq 0 \\
0 & <0 & \dots      &\geq 0 \\
\vdots & \vdots & <0 &\geq 0 \\
0 & 0 & \dots &0 \end{array} \right)
\]
That is, if $A_k = (a^k_{ij})_{n\times n}$ then
\[
a_{ij}^k \begin{cases}
\geq 0 \text{ if } i>j\\
= 0 \text{ if } i<j \text{ or } i=n\\
<0 \text{ if } i=j\neq n
\end{cases}
\]
\end{lemma}
\begin{proof}
Since $A_{k-1}$ is diffusive, and $A_k$ is not, then as in the proof of Lemma~\ref{lem:GE1}, $A_k$ must have been obtained from $A_{k-1}$ by an operation
\[
\op{row}^k_q := \op{row}^{k-1}_q + c\cdot \op{row}^{k-1}_p
\]
where $c>0$ and $p,q\in\{1,2,\dots,n\}$ are distinct rows. If $A_k = (a^k_{ij})_{n\times n}$ then
\[
a^k_{qj} = a^{k-1}_{qj} + c a^{k-1}_{pj}
\begin{cases}
= 0 \text{ if } j=p \text{ by choice of } c \text{ in Gaussian elimination;}\\
\geq 0 \text{ if } j\neq q, j\neq p \text{ since } a^{k-1}_{qj}\geq 0, a^{k-1}_{pj}\geq 0 \text{ by diffusivity of } A_{k-1};\\
\geq 0 \text{ if } j=q \text{ since } A_k \text{ is not diffusive, and }A_{k-1}\text{ is.}
\end{cases}
\]
Therefore $\op{row}^k_q \in\R^n_{\geq 0}$. By assumption, $\op{row}^k_q = 0$.

Suppose for contradiction that $q\neq n$. Consider the $n$'th column. Note that $a^0_{in} >0$ for all $i\neq n$ since $A^0 = A$ is strongly diffusive. It follows by induction that for all $i\neq n$, $a^1_{in},a^2_{in},\dots, a^k_{in} >0$. In particular, $a^k_{qn} > 0$ contradicting $\op{row}^k_q = 0$. Hence $q = n$.

The fact that $A_k$ is upper-triangular is a consequence of  Gaussian elimination reaching the last row. The sign pattern follows because $A_{k-1}$ was diffusive.
\end{proof}

\begin{corollary}\label{cor:fullrank}
Let $A$ be an $(n-1)\times n$ strongly diffusive matrix.\\ If ${\op{cone}(\op{rows}(A)) \cap \R^n_{\geq 0 }=\{0\}}$ then $A$ has rank $n-1$.
\end{corollary}
\begin{proof}
Observe that Gaussian elimination does not employ the $n$'th row of an $n\times n$ matrix to put the $(n-1)\times n$ submatrix obtained by dropping the $n$'th row into upper triangular form.

Suppose ${\op{cone}(\op{rows}(A)) \cap \R^n_{\geq 0 }=\{0\}}$. Consider applying Gaussian elimination to $A$. After each step, either the matrix remains diffusive, or a row becomes non-negative, by Lemma~\ref{lem:GE1}. If we never get a non-negative row, Gaussian elimination goes all the way to the end, and gives us a matrix in row echelon form, which clearly has rank $n-1$. If we get a non-negative row, by assumption the row must be $0$. However, this contradicts Lemma~\ref{lem:GE2}, and hence this case can never happen.
\end{proof}

\begin{theorem}[Convex Rank-Nullity Theorem]\label{rank-nullity}
Let $A$ be a strongly diffusive $n\times n$ matrix. Then either $\ker{A} \cap \R^n_{\geq 0} \neq \{0\}$ or $\op{cone}(\op{rows}(A)) \cap \R^n_{>0}$ is nonempty or $\R^n_{<0} \subseteq \op{cone}(\op{rows}(A))$.
\end{theorem}
\begin{proof}
Suppose $\op{cone}(\op{rows}(A)) \cap \R^n_{>0} = \emptyset$ and $\ker{A} \cap \R^n_{\geq 0} = \{0\}$.

By Lemma~\ref{positive_cone} we must have $\op{cone}(\op{rows}(A))\cap\R^n_{\geq 0} = \{0\}$. Hence applying Corollary~\ref{cor:fullrank} to the first $n-1$ rows of $A$, it follows that $A$ has rank at least $n-1$.

We next argue that $A$ has rank $n$. Suppose for contradiction that $A$ has rank $n-1$. Let $v\in\ker{A}$ be a non-zero vector. Since $\ker{A} \cap \R^n_{\geq 0} = \{0\}$, we know that $v\notin \R^n_{\geq 0}$. Reordering indices if necessary, we may assume without loss of generality that $v$ has the sign pattern $(\geq 0,\geq 0, \dots, \geq 0, >0, \leq 0,\leq 0,\dots, \leq 0)$, with $k$ such that $v_k > 0$ and $v_{k+1} \leq 0$. Now applying Gaussian elimination to $A$ with this ordering leads to an upper triangular matrix $B$ as in Lemma~\ref{lem:GE2} with the last row $0$. In particular, the $k$'th row $r_k^B$ of this matrix has the sign pattern $(0,0,\dots,0,<0,\geq 0,\dots,\geq 0)$ with the entry in position $k$ having the only negative sign. Since $v\in\ker{A} = \ker {B}$, we must have $r_k^B\cdot v = 0$. But from the sign analysis, $r_k^B\cdot v < 0$, which is a contradiction. Hence, $A$ has rank $n$.

The sequence of matrices produced by Gaussian elimination applied to $A$ are all diffusive. This is because $\op{cone}(\op{rows}(A)) \cap \R^n_{\geq 0} = \{0\}$ and $A$ has rank $n$, so Lemma~\ref{lem:GE1} applies inductively.
So the final matrix $B$ output by Gaussian elimination is diagonal and diffusive. That is:
\[
B = \left(
\begin{array}{cccc}
<0 &  0 & \dots      & 0 \\
0 & <0 & \dots      & 0 \\
\vdots & \vdots & <0 & 0 \\
0 & 0 & \dots &<0
\end{array}
\right)
\]
By repeated application of Lemma~\ref{lem:GE1}, ${\op{cone}(\op{rows}(B)) \subseteq \op{cone}(\op{rows}(A))}$,\\and so $\R^n_{<0}\subseteq \op{cone}(\op{rows}(A))$.
\end{proof}

\begin{remark}\label{rmk:notfarkas}
Theorem~\ref{rank-nullity} might remind some readers of Gordan's theorem~\cite[(31), pp.~95]{schrijver}, which asserts unconditionally that either $\ker{A} \cap \R^n_{\geq 0} \neq \{0\}$ or $\op{span}(\op{rows}(A)) \cap \R^n_{>0}$ is nonempty. However, the two theorems are different, and it is instructive to compare the differences.

Theorem~\ref{rank-nullity} makes a stronger assumption --- that $A$ is strongly-diffusive --- and produces a stronger implication: if the kernel is trivial then the cone (as opposed to span in Gordan's theorem) of the rows meets either the positive or the negative orthant. In particular, there are matrices that violate the stronger assumption for which the stronger conclusion is also false. For example, consider the diffusive, but not strongly diffusive, matrix
\[
\left(
\begin{array}{ccc}
-1 & 0 & 0\
\\0 & -1 & 2\
\\0 & 1 & -1
\end{array}
\right).
\] The conclusion of Theorem~\ref{rank-nullity} is false for this matrix. This is because the kernel is zero. The cone of the rows does not contain any strictly positive or strictly negative vector, as can be easily verified.
\end{remark}

\section{The Main Result}\label{sec:main}

\begin{lemma}\label{closure}
Let $G=(\SS,\RR)$ be a positive reaction network. Let $T\subseteq\SS$ be a minimal siphon. Then for all $i\in T$, the closure $\op{Cl}(\{i\}\cup \SS\setminus T) = \SS$.
\end{lemma}
\begin{proof}
Note that $\SS\setminus T\subsetneq \{i\}\cup \SS\setminus T \subseteq \op{Cl}(\{i\}\cup \SS\setminus T)$. Therefore,
the complement $\SS\setminus \op{Cl}(\{i\}\cup \SS\setminus T) \subsetneq T$. But $\SS\setminus \op{Cl}(\{i\}\cup \SS\setminus T)$ is the complement of a closed set, and hence must be a siphon by Remark~\ref{Rmk:1}.\ref{closedsiphon}. Since $T$ is a minimal siphon, a siphon properly contained in $T$ can only be the empty set. Hence, $\op{Cl}(\{i\}\cup \SS\setminus T) = \SS$.
\end{proof}

\begin{lemma}\label{minimal-diffusive}
Let $G=(\SS,\RR)$ be a positive reaction network. Suppose $T\subseteq\SS$ is a minimal siphon that is not self-replicable. Then there exists a strongly diffusive $T\times T$ matrix $A$ such that $\op{cone}(\op{rows}(A))\subseteq\op{cone}(\op{rows}(\Gamma_T))$ and $\ker{A} = \ker{\Gamma_T}$.
\end{lemma}
\begin{proof}
By Lemma~\ref{closure}, we have $\op{Cl}(\{i\}\cup \SS\setminus T) = \SS$ for every $i\in T$. Hence for every $i \in T$, there exists a reaction pathway $y(i) \rpath_G y'(i)$ such that $\supp(y(i))=i\cup \SS\setminus T$ and $\supp(y'(i))=\SS$. Define the matrix $A=(y'(i)_j-y(i)_j)_{T\times T}$.

To see that $A$ is strongly diffusive, first note that $y'(i)_j-y(i)_j > 0$ for all $i,j\in T$ such that $i\neq j$ since $y(i)_j = 0$. Next note that if there exists $i\in T$ such that $y'(i)_i-y(i)_i\geq 0$, then Lemma~\ref{positive_cone} implies that $T$ is self-replicable. Hence $y'(i)_i-y(i)_i < 0$ for all $i\in T$.

Since the rows of $A$ correspond to reaction pathways, Lemma~\ref{lem:reactionpathway} implies that\ \\${\op{cone}(\op{rows}(A))\subseteq\op{cone}(\op{rows}(\Gamma_T))}$.

It is clear that $\ker{\Gamma_T}\subseteq\ker{A}$. Suppose the containment is strict. Then there exists a row of $\Gamma_T$ that is not in the span of the rows of $A$. In particular $A$ does not have full rank, so by Corollary~\ref{cor:fullrank}, $A$ must have rank $n-1$ and the first $n-1$ rows of $A$ are linearly independent. Further, $\Gamma_T$ has a row that is not in the span of the rows of $A$. Adding $\epsilon$ times this row of $\Gamma_T$ to the last row of $A$, where $\epsilon > 0$ is sufficiently small, gives a new matrix $\tilde{A}$ that is full rank (hence $\ker{\Gamma_T} = \ker{\tilde{A}}=\{0\}$), strongly-diffusive, and such that $\op{cone}(\op{rows}(\tilde{A}))\subseteq\op{cone}(\op{rows}(\Gamma_T))$.
\end{proof}

\begin{theorem}\label{minimal-siphon}
Let $G=(\SS,\RR)$ be a positive reaction network. Let $\emptyset\neq T\subseteq S$. Then
\begin{enumerate}
\item\label{drsrthencrit} If $T$ is either drainable or self-replicable then it is critical.
\item If $T$ is a minimal critical siphon then it is either drainable or self-replicable.
\item\label{wrequivcds} If $G$ is consistent then the following are equivalent:
  \begin{enumerate}
  \item $T$ is critical.
  \item $T$ is drainable.
  \item $T$ is self-replicable.
  \end{enumerate}
\end{enumerate}
\end{theorem}

\begin{proof}
1. For contradiction, suppose $T$ is not critical. Then there exists a positive conservation law $w.x$ with support in $T$. If $T$ is self-replicable then there exists a reaction pathway $y \rpath_G y'$ such that $y'_i - y_i > 0$ for all $i\in T$. In particular, $w.y' > w.y$, contradicting Lemma~\ref{lem:reactionpathway} that $w.x$ is invariant along $G$-reaction pathways. If $T$ is drainable, we get a similar contradiction from $y'_i - y_i < 0$ for all $i\in T$ . Hence, $T$ is critical.\
\\
\\2. Let $T$ be a minimal critical siphon. Since $T$ is critical, by Theorem~\ref{critical-equivalence}, we have $\ker{\Gamma_T} \cap \R^T_{\geq 0} = \{0\}$. Since $T$ is a minimal siphon, Lemma~\ref{minimal-diffusive} gives a strongly diffusive matrix $A$ such that $\op{cone}(\op{rows}(A))\subseteq \op{cone}(\op{rows}(\Gamma_T))$ and $\ker{A} = \ker{\Gamma_T}$.

Applying Theorem~\ref{rank-nullity} to the matrix $A$, either
\[
{\op{cone}(\op{rows}(\Gamma_T))\cap \R^n_{>0}\supseteq\op{cone}(\op{rows}(A))\cap \R^n_{>0}\neq\emptyset}
\]
so that $T$ is self-replicable, or
\[
{\op{cone}(\op{rows}(\Gamma_T))\cap \R^n_{<0}\supseteq \op{cone}(\op{rows}(A))\cap \R^n_{<0}\neq\emptyset}
\]
so that $T$ is drainable.
\\
\\3. The implication $(b)\Rightarrow (a)$ is true by part 1. Now suppose that $G$ is consistent. We will show $(a)\Rightarrow(c)$ and $(c)\Rightarrow (b)$.\
\\
Since $G$ is consistent, for each reaction $y\to y'\in\RR$ there is a number $a_{y\ra y'}> 0$  such that
\begin{align}
\sum_{y\ra y'\in\RR}a_{y\ra y'}(y'- y) = 0.
\end{align}
In particular, for all $i \in T$, we have
\begin{align}\label{1}
\sum_{y\ra y'\in\RR}a_{y\ra y'}(y'-y)_i = 0.
\end{align}
Since $\R^S_{>0}$ is an open set, we may assume the $a_{y\ra y'}$ are rational numbers. Normalizing denominators, we may assume the $a_{y\ra y'}$ are integers.
\\
\\$(a)\Rightarrow (c)$:
Suppose $T$ is critical. We will show that $T$ is self-replicable. Towards this, let $z\in\mathbb{Z}^\SS_{\geq 0}$ be such that $z_i=0$ precisely when $i\in T$. Then $z$ is a critical point by Theorem~\ref{critical-equivalence}. From the definition of a critical point, there exist reals $e_{y\ra y'}$ such that
\begin{equation}
z+\sum_{y\ra y'\in\RR}e_{y\ra y'}(y'-y)\in\R^S_{>0}.
\end{equation}
In particular, for all $i\in T$, since $z_i=0$, we have
\begin{align}\label{3}
\sum_{y\ra y'\in\RR}e_{y\ra y'}(y'-y)_i > 0.
\end{align}
Since $\R^S_{>0}$ is an open set, we may assume the $e_{y\ra y'}$ are rational numbers. Normalizing denominators, we may assume the $e_{y\ra y'}$ are integers. Now choose strictly positive integers $f_{y\ra y'}$ such that $f_{y\ra y'}\,a_{y\ra y'}+ e_{y\ra y'} > 0$. Let $n = \max_{y\ra y'}\{f_{y\ra y'}\}$. Adding $n$ times Equation ~\ref{1} to Equation ~\ref{3} yields for all $i \in T$:

\[
\sum_{y\ra y'\in\RR}\left( n\,a_{y\ra y'}+ e_{y\ra y'}\right)(y'-y)_i > 0.
\]
Hence $T$ is self-replicable.
\\
\\$(c)\Rightarrow (b)$ Suppose $T$ is self-replicable. Then there exist positive integers $b_{y\ra y'}$ such that for all $i\in T$
\begin{equation}\label{2}
\sum_{y\ra y'\in\RR}b_{y\ra y'}(y'-y)_i >0.
\end{equation}
We choose strictly positive integers $c_{y\ra y'}$ such that $c_{y\ra y'}\,a_{y\ra y'}- b_{y\ra y'} > 0$, and define $m := \max_{y\ra y'}\{c_{y\ra y'}\}$. Subtracting Equation~\ref{2} from $m$ times Equation~\ref{1}  yields for all $i\in T$:
\[
\sum_{y\ra y'\in\RR}\left(m\,a_{y\ra y'} - b_{y\ra y'}\right)(y'-y)_i < 0.
\]
Hence $T$ is drainable.\
\end{proof}

\begin{example}\label{ex:cds}
Some critical siphons are neither self-replicable nor drainable. Consider the following network: \[X\rightarrow 2Y, \mbox{  }2X\rightarrow Y \]
In this network, the siphon $\{X,Y\}$ is critical. However, $\{X,Y\}$ is neither self-replicable nor drainable.

A self-replicable siphon need not be a minimal critical siphon. Consider the following network: \[X\rightarrow 2X, \mbox{  }2X\rightarrow 2X+Y \]
The siphon $\{X,Y\}$ is self-replicable because there exists a reaction pathway $X\rpath 2X+Y$, and therefore $\{X,Y\}$ is a critical siphon by Theorem~\ref{minimal-siphon}.\ref{drsrthencrit}. However, it is not a minimal critical siphon because $\{X\}$ is also a critical siphon.
\end{example}

\section{Persistence}

\begin{definition}
A reaction network $G = (\SS,\RR)$ is
{\em persistent} iff for every choice $k:\RR\to\R_{>0}$ of specific rates, for all initial conditions $x(0)\in\R^\SS_{>0}$, for the mass-action ODE system
\[
	\dot{x}(t) = \sum_{y\ra y'\in \RR} k_{y\ra y'}(y'- y) x(t)^y ,
\]
the omega-limit set $\omega(x)$ does not meet $\partial\R^\SS_{\geq 0}$, i.e., for every increasing sequence $0 < t_1 < t_2 <\dots$ of positive real numbers tending to $+\infty$,
the limit $\displaystyle\lim_{i\to\infty} x(t_i) \notin\partial\R^\SS_{\geq 0}$.
\end{definition}

The definition of persistence is standard in dynamical systems, and we have adapted it here for mass-action systems.

\begin{theorem}\label{thm:persistence}
Let $G=(\SS,\RR)$ be a positive reaction network. If $G$ has no drainable siphons then $G$ is persistent.
\end{theorem}
\begin{proof}
We will prove the contrapositive. Suppose $G$ is not persistent. We will first identify a siphon in $G$, and then show that siphon to be ``drainable.'' There exist a point ${x_0\in\R^\SS_{>0}}$ and specific rates $k:\RR\to \R_{>0}$ such that the solution $x(t)$ to the mass-action ODE system with initial conditions $x(0) = x_0$ violates persistence, so that the omega-limit set $\omega(x)$ meets the boundary $\partial\R^\SS_{\geq 0}$ of the positive orthant. Let $z\in\omega(x)\cap\partial
R^\SS_{\geq 0}$. From \cite[Proposition~1]{Angeli}, we know that the set
\[
T_z:=\{i\mid z_i = 0\}
\]
is a siphon. (Strictly speaking, \cite[Proposition~1]{Angeli} is in the context of chemical reaction networks, but the proof applies without change to positive reaction networks.)

We now show that $T_z$ is drainable. Define $\op{Cone}_\RR := \op{Cone}(\{y'-y\mid y\ra y'\in\RR\})$.

We first claim that $z\in x(0) + \op{Cone}_\RR$. By definition of $\omega$-limit point, there exists an increasing sequence $0 < t_1 < t_2 <\dots$ of positive real numbers tending to $+\infty$ such that $z=\displaystyle\lim_{i\to\infty} x(t_i)$. Further, writing the integral of the mass-action ODE system as a Riemann sum,
\[
x(t) = x(0) + \lim_{h\to 0} \sum_{j=0}^{t/h-1} \sum_{y\ra y'\in\RR} h\cdot k_{y\ra y'}\cdot(y'-y)\cdot (x(jh))^y,
\]
so that
\[
z = x(0) + \lim_{i\to\infty} \lim_{h\to 0} \sum_{j=0}^{t_i/h-1} \sum_{y\ra y'\in\RR} h\cdot k_{y\ra y'}\cdot(y'-y)\cdot (x(jh))^y.
\]
Each of the partial sums
\[z_{i,h} = x(0) + \sum_{j=0}^{t_i/h-1} \sum_{y\ra y'\in\RR} h\cdot k_{y\ra y'}\cdot(y'-y)\cdot (x(jh))^y\]
is in $\op{Cone}_\RR$ because $h,k_{y\ra y'}, (x(jh))^y>0$; the limit exists by choice of $t_i$; and $\op{Cone}_\RR$ is a closed set by Definition~\ref{def:cone}. Hence, $z\in x(0) + \op{Cone}_\RR$.

So there exist non-negative real numbers $a_{y\ra y'}\in\R_{\geq 0}$ such that
\[
z = x(0) + \sum_{y\ra y'\in\RR} a_{y\ra y'} (y'-y).
\]
Rewriting this, we get $x(0) = z - \sum_{y\ra y'\in\RR} a_{y\ra y'} (y'-y)\in\R^\SS_{>0}$ by assumption. Since $\R^\SS_{>0}$ is open, we can slightly perturb the numbers $a_{y\ra y'}$ if necessary so that they lie in the non-negative rational numbers, while ensuring that $z - \sum_{y\ra y'\in\RR} a_{y\ra y'} (y'-y)$ remains in the positive orthant. Now clearing denominators of the numbers $a_{y\ra y'}$ by multiplying by a large positive integer $L$, we get that $\alpha := L z - \sum_{y\ra y'\in\RR} b_{y\ra y'} (y'-y)\in\R^\SS_{>0}$, where $b_{y\ra y'} = L a_{y\ra y'}$ are all non-negative integers. In other words, there is a $G$-reaction pathway from $\alpha\in\R^\SS_{>0}$ to $Lz$, showing that $T_{Lz} = T_z$ is drainable.
\end{proof}

\begin{remark}
Let $G=(\SS,\RR)$ be a positive reaction network. If $T_1\subseteq\SS$ is a drainable siphon and $T_2\subseteq T_1$ is a siphon then $T_2$ is drainable by Remark~\ref{Rmk:1}.\ref{subset-drainable}. In particular, every minimal drainable siphon is a minimal siphon. Hence to show that a positive reaction network is persistent, it is sufficient to show that its {\em minimal} siphons are not drainable.
\end{remark}

The next remark provides a vague indication of future directions for attack on the Global Attractor Conjecture prompted by the results in this paper.

\begin{remark}\label{Rmk:compete}
Many reaction networks with drainable siphons are nevertheless persistent. One natural question to explore is precisely when do minimal drainable siphons cause extinction? If a  minimal drainable siphon is simultaneously also self-replicable, then intuitively it seems as if there is a competition between extinction and autocatalytic growth. Let us say, very informally, that the self-replicable nature of a minimal critical siphon {\em dominates} the drainable nature if this competition is won by autocatalytic growth, so that there is no extinction of the species in $T$.

We will leave this notion of ``domination'' vague at the moment; making it precise in combinatorial terms, and linking the combinatorics to the dynamics, are two of the challenges in taking this work forward. In this regard, also see notions of draining, sustaining, and domination in \cite[Proposition~6.26]{geometricgac} which are cognate to the present notions of drainable, self-replicable, and domination respectively.

Note that in weakly-reversible networks, every drainable siphon is also self-replicable by Theorem~\ref{minimal-siphon}.\ref{wrequivcds}. Perhaps the definition of domination, when provided, will allow us to reformulate the persistence conjecture thus: in weakly-reversible networks, for minimal critical siphons, the self-replicable nature dominates the drainable nature.
\end{remark}

\section{Catalytic sets and Catalysis}\label{sec:catalysis}%

The goal of this section is to give an elementary proof to Theorem~\ref{connect}, which was first proved in \cite{gopalkrishnan} with algebraic geometric techniques. We recall some definitions from \cite{gopalkrishnan}.

\begin{definition}[Definition~3.1 in \cite{gopalkrishnan}]
Let $G=(\SS,\RR)$ be a weakly-reversible chemical reaction network. The {\em event-graph} $\overline{G}$ of $G$ is a directed graph with vertices $V(\overline{G}) = \ZZ^\SS_{\geq 0}$ and edges the $G$-dilutions $y\ra y'$ where both $y,y'\in\ZZ^\SS_{\geq 0}$.
\end{definition}

\begin{remark}
In \cite{gopalkrishnan}, the event-graph was defined in terms of the monomials $\prod_{i\in \SS} x_i^{y_i}$ and $\prod_{i\in\SS}x_i^{y_i'}$ because the techniques in that paper came from algebraic geometry. Here we have rephrased the same definition in terms of dilutions. The equivalence is immediate from comparing the two definitions.
\end{remark}

The next lemma bridges positive and chemical reaction networks. Specifically, it makes the point that in chemical reaction networks, it is sufficient to consider complexes that have non-negative integer coordinates, as opposed to non-negative real coordinates.

\begin{lemma}\label{lem:poschem}
Let $G=(\SS,\RR)$ be a chemical reaction network. If $y\rpath_G y'$ is a reaction pathway then
\begin{enumerate}
\item $y-\lfloor y\rfloor = y'- \lfloor y'\rfloor$.
\item There is a reaction pathway $\lfloor y \rfloor \rpath_G \lfloor y'\rfloor$.
\end{enumerate}
\end{lemma}
\begin{proof}
Since $G$ is chemical, each $G$-reaction (and hence each $G$-dilution) can only change the population by an integer vector, and fractional parts of complexes are carried around unchanged by reaction pathways, which establishes part 1.

To see part 2, suppose $y\rpath_G y'$ corresponds to the sequence of dilutions\
\\${y = y(0) \ra y(1) \ra \dots \ra y(m) = y'}$. Then by part 1, since every dilution is a reaction pathway, $y(i) - \lfloor y(i) \rfloor = y - \lfloor y \rfloor$ for all $i$. Since $G$ is chemical, removing the fractional part of a complex does not affect the reactions that can occur with that complex as source. So ${\lfloor y(0) \rfloor \ra \lfloor y(1) \rfloor \ra \dots \ra \lfloor y(m) \rfloor}$ is a sequence of dilutions, and there is a reaction pathway $\lfloor y \rfloor \rpath_G \lfloor y'\rfloor$.
\end{proof}

The next lemma restates the definitions of catalytic and strictly catalytic chemical reaction networks which were introduced in \cite{gopalkrishnan} in the language of reaction pathways. Readers who do not wish to consult \cite{gopalkrishnan} at this point may also consider the next lemma as the definition of catalytic and strictly-catalytic chemical reaction networks, and skip the proof.

\begin{lemma}\label{lem:cat}
A weakly-reversible chemical reaction network $G$ is catalytic (respectively, strictly catalytic) iff there exist complexes $y,y'\in\ZZ^S_{\geq 0}$ such that
\[
y\rpath_G y'\text{ but }{y-\min(y,y')\notrpath_G y'-\min(y,y')}
\]
(respectively, $y\rpath_G y'$ but for all $k\in \Rplus$, $k(y-\min(y,y'))\notrpath_G k(y'-\min(y,y'))$).
\end{lemma}
\begin{proof}
Since edges in $\overline{G}$ correspond to dilutions, complexes $y,y'$ are path-connected in $\overline{G}$ iff $y\rpath_G y'$. Next note that $\gcd(x^y,x^{y'}) = x^{\min(y,y')}$. The catalytic part of the lemma now follows from Definition~3.2 in \cite{gopalkrishnan}. To see the strictly catalytic part, first observe that by Definition~6.1 in \cite{gopalkrishnan}, $G$ is strictly catalytic iff there exist complexes $y,y'\in\ZZ^S_{\geq 0}$ such that $y\rpath_G y'$ and $k(y-\min(y,y'))\notrpath_G k(y'-\min(y,y'))$ for all $k\in \ZZ_{\geq 1}$.

It remains to show that the definition does not change if we quantify over all $k\in\Rplus$. One direction is easy. For the other direction, we proceed by contrapositive. Suppose $y,y'\in\ZZ^S_{\geq 0}$ are such that $y\rpath_G y'$ and there exists $k\in\Rplus$ such that $k(y-\min(y,y'))\rpath_G k(y'-\min(y,y'))$. By dropping any fractional part if necessary as in Lemma~\ref{lem:poschem}, we may assume that  $k(y-\min(y,y'))$ is an integer vector, so that $k$ is rational. Let $p,q$ be positive integers such that $k=\frac{p}{q}$. Then we get a reaction pathway $kq(y-\min(y,y'))\rpath_G kq(y'-\min(y,y'))$ by applying the above reaction pathway $q$ times. Since $kq=p$ is a positive integer, we are done.
\end{proof}

The intuition behind the definition of ``catalytic'' is that $\min(y,y')$ prevents the conversion of $y - \min(y,y')$ to $y'- \min(y,y')$. ``Strictly catalytic'' indicates that the conversion is not possible even with multiple copies of $y-\min(y,y')$.

Note that we have only defined what it means for a network to be (strictly) catalytic, but have said nothing about which species may be identified as the catalysts. The next definition makes precise this identification of the catalytic species, in the process also extending the concept of catalysis beyond weakly-reversible chemical reaction networks, to all positive reaction networks.

\begin{definition}[Catalytic set]\label{def:catalyticset}
Let $G=(\SS,\RR)$ be a positive reaction network. 
A set $T\subseteq\SS$ is {\em catalytic} iff there is a reaction pathway $y\rpath_G y'$ such that
\begin{enumerate}
\item supp($min(y,y'))=T$.
\item $y-\min(y, y')\notrpath_G y'-\min(y,y'))$.
\end{enumerate}
It is {\em strictly-catalytic} if in addition $k(y-\min(y, y'))\notrpath_G k(y'-\min(y,y'))$ for every $k\in\Rplus$.
\end{definition}

We immediately obtain the next theorem as an easy consequence.

\begin{theorem}\label{catset-catalytic}
Let $G$ be a weakly reversible chemical reaction network. Then $G$ has a catalytic (respectively, strictly-catalytic) set $T\subseteq\SS$ iff $G$ is catalytic (respectively, strictly-catalytic).
\end{theorem}
\begin{proof}
Immediate from Lemma~\ref{lem:poschem}, Lemma~\ref{lem:cat} and Definition~\ref{def:catalyticset}.
\end{proof}

This theorem allows us to extend the previous definition of catalytic and strictly-catalytic networks from the class of weakly-reversible chemical reaction networks to all positive reaction networks.

\begin{definition}
A positive reaction network $G=(\SS,\RR)$ is {\em catalytic} (respectively {\em strictly catalytic}) iff there exists a catalytic (respectively, strictly catalytic) subset $T\subseteq\SS$.
\end{definition}

\begin{example}
For the reaction network:
\begin{center}
$3X \rightarrow 0$\\
$X+2Y \rightarrow 2Y$,\\
\end{center}
the set $\{Y\}$ is catalytic because there exists a reaction pathway $X+2Y\rpath_G 2Y$, but there is no reaction pathway $X \rpath_G 0$. However, it is not strictly catalytic, because there is a reaction pathway $3X\rpath_G 0$.

For the reaction network:
\begin{center}
$X+Y \rightarrow Y$\\
$X+2Y \rightarrow 2Y$,\\
\end{center}
the set $\{Y\}$ is strictly catalytic here because there exists a reaction pathway ${X+2Y\rpath_G 2Y}$, and for all $k\in\Rplus$, $kX \notrpath_G 0$, since the absence of $Y$ turns off all reactions.
\end{example}

The next theorem shows how the notions of self-replicable, strictly-catalytic, and siphon are related.

\begin{theorem}\label{thm:autocatalytic}
Let $G=(\SS,\RR)$ be a positive reaction network. If $T\subseteq\SS$ is a self-replicable set that is not strictly-catalytic then no nonempty subset of $T$ is a siphon.
\end{theorem}
\begin{proof}
Suppose $T$ is a self-replicable set that is not strictly-catalytic. Since $T$ is self-replicable, there exists a reaction pathway $y\rpath_G y'$ such that $y'_i - y_i > 0$ for all $i\in T$. Since $T$ is not strictly-catalytic, there exists $k\in\Rplus$ such that $k(y - \min(y,y')) \rpath_G k(y'-\min(y,y'))$. Note that for all $i\in T$, $k(y_i - \min(y_i,y_i')) = 0$ whereas $k(y_i'-\min(y_i,y_i')) > 0$. Hence no nonempty subset of $T$ is a siphon by Lemma~\ref{lem:siphonrpath}.\
\end{proof}

\begin{example}\label{counter}
The converse of the above theorem is false. Consider the following reaction network:
\begin{center}
$X \rightarrow 2X$\\
$Y \rightarrow X+Y$\\
\end{center}
In this example, no non-empty subset of $X$ is a siphon. Notice that $X$ is self-replicable because of the reaction pathway $X \rightarrow 2X$ that increases the number of $X$. But $X$ is strictly catalytic as we have a reaction pathway $X \rightarrow 2X$, but no reaction pathway $0 \rightarrow kX$ for every $k\in\R_{>0}$.
\end{example}

\begin{corollary}\label{cor:srtosc}
Let $G=(\SS,\RR)$ be a positive reaction network. If $T\subseteq\SS$ is a self-replicable siphon then $T$ is a strictly-catalytic set.
\end{corollary}

\begin{corollary}\label{cor:crittosc}
Let $G=(\SS,\RR)$ be a weakly-reversible positive reaction network. If $T\subseteq\SS$ is a critical siphon then $T$ is a strictly-catalytic set.
\end{corollary}
\begin{proof}
By Theorem~\ref{minimal-siphon}.\ref{wrequivcds}, Remark~\ref{wr-cons} and Corollary~\ref{cor:srtosc}
\end{proof}

From Corollary~\ref{cor:crittosc} and Theorem~\ref{catset-catalytic}, we have the following theorem.
\begin{theorem}\label{thm:connectstrict}
Weakly-reversible chemical reaction networks with critical siphons are strictly catalytic.
\end{theorem}

Since every strictly catalytic network is also catalytic we also obtain an elementary proof for the following theorem first proved in \cite{gopalkrishnan}.

\begin{theorem}\label{connect}
Weakly-reversible chemical reaction networks with critical siphons are  catalytic.
\end{theorem}

Note that \cite{gopalkrishnan} made a connection between saturated binomial ideals and catalysis, and used the algebraic geometry of binomial ideals to obtain this result.

\begin{example}
The converse of the above theorem is false. Consider the reaction network in Example~\ref{counter}. In this case, there is a reaction pathway $Y \rpath X+Y$, but there is no reaction pathway $0\rpath_G X$. Hence the set $Y$ is catalytic. But the above network does not possess any self-replicable siphons. The species $X$ is not a siphon because of the second reaction. The sets $\{Y\}$ and $\{X,Y\}$ are not critical because of the conservation law $y$.
\end{example}

\section{Related Work}
A reaction network admits various choices for dynamics whose appropriateness depends on the context of the application. We briefly describe some of the possible choices.
\begin{bullets}

\item {\em Mass-action kinetics}~\cite{WaageGuldberg} prescribes a way for going from a reaction network and a collection of reaction ``specific rates'' to a system of ordinary differential equations.  The variables of this ODE system are concentrations of the reaction species as functions of time, and the right-hand sides are polynomials in these variables. Many of the mathematical results about chemical reaction networks have been proved in the setting of mass-action kinetics. Such models also appear in the study of power-law systems and S-systems from biochemical systems theory~\cite{bio_systems}. Quadratic dynamical systems~\cite{RabinovichSinclairWigderson} studied in computer science to model genetic algorithms are mathematically cognate to bimolecular mass-action systems. Many models in population biology like predator-prey systems~\cite{Lotka, Volterra} are described as mass-action systems.

\item The {\em chemical master equation} describes the evolution of species populations as a continuous-time Markov process, given a ``specific propensity'' for each reaction. It is a physically more accurate choice for dynamics than mass-action kinetics, since it also takes into account stochastic effects. It is especially relevant when modeling reactions where the species population sizes are small, and the stochastic effects can cause a significant deviation from mass-action behavior.

\item A common choice of dynamics in the computer science community, especially in the study of Petri nets, is {\em non-deterministic dynamics}. The key question here is that of ``reachability:'' whether there is a path to go from one population vector to another, by applying the reactions in any sequence.

\item In practice, the rates and propensities of reactions are often not easily obtainable by experiment. Indeed, they may not even be constants, but may vary in time, or across assays. This can be captured by {\em mass-action differential inclusions}~\cite{projectionpaper}, where the reaction rates simultaneously take all values in a fixed interval bounded away from $0$ and $+\infty$.

\item The assumption of mass-action can be relaxed to allow other functional forms for the right-hand sides of the ODEs. Many of the results first proved in the setting of mass-action can be extended to these more general dynamical systems~\cite{banaji}.

\item Stochastic and deterministic {\em reaction-diffusion} systems combine spatial effects of diffusion with some local dynamical model of the progress of reactions. They are relevant in settings where the assumption that the reactants are well-mixed no longer holds.
\end{bullets}

Reaction networks have been studied in chemistry for a long time~\cite{WaageGuldberg}, but they have also been studied beyond chemistry. In 1962, Carl Petri~\cite{PetriThesis} applied reaction networks to the study of distributed processes in computer science, and they have come to be known in this area as ``Petri nets.'' Kauffman~\cite{Kauffman} has conjectured a key role for autocatalytic reaction networks in the origin of life.

In recent times, advances in systems and synthetic biology and molecular computation have given a new impetus to the study of reaction networks. It is now commonplace wisdom in biology that the sophisticated behavior of living systems emerges from the dynamics of their reaction networks. As high-throughput techniques uncover more details about these biochemical reaction networks, the questions of inference and control of dynamical properties of the system from this combinatorial data becomes more compelling. Such efforts are already underway in the areas of Flux Balance Analysis~\cite{flux_balance1,flux_balance2}, and Biochemical Systems Theory~\cite{bio_systems}. It is hoped that a deep mathematical understanding of reaction networks would aid such efforts.

In the area of molecular computation, researchers have engineered proof-of-principle molecular circuits that synthesize sophisticated behavior~\cite{SeesawQianWinfree}. Interesting applications would result if these constructions could be scaled economically by two orders of magnitude. One problem with attempts at scaling has been ``retroactivity:'' composition of two reaction networks leads to changes in the behavior of each network~\cite{SontagRetroactivity}, so that the gestalt dynamics does not have the targeted behavior. To obtain a good compositional scheme for reaction networks, it may be helpful to investigate whether there are mathematically-natural ways for reaction networks to be composed, and decomposed.

The persistence conjecture~\cite{Feinberg89} is a good candidate problem for gauging progress in our understanding of the dynamics of reaction networks. In recent times, there has been a resurgence of interest in this conjecture~\cite{CNP, Pantea,Anderson11, gopalkrishnan, projectionpaper, geometricgac}, and in critical siphons~\cite{Anderson08,Anderson10,shiu,gopalkrishnan}.

Autocatalysis almost certainly plays a central role in the complexity and self-organization exhibited by living systems, and possibly also played a role in the origin of life~\cite{Kauffman}. Autocatalysis is also a feature of synthetic DNA molecular circuits that perform signal amplification~\cite{PengHairpinProgrammingBiomolecular, EntropyDrivenZhang} --- signal amplification is essential for the scaling of molecular circuit constructions.

Our notion of ``self-replicable siphons'' may be relevant in the context of Kauffman's program~\cite{Kauffman} of explaining the origin of life using {\em autocatalytic sets of reactions}. Autocatalytic sets of reactions have been studied extensively~\cite{steel,hordijk_steel,mossel_steel,hordijk_hein_steel,hordijk_kauffman_steel,steel_hordijk,hordijk_steel_kauffman,sanjay_jain,pnas_jain,varun_giri,kauffman,eigen,dyson}. The notion of self-replicable siphons~(Definition~\ref{def:sr}) appears to be cognate with the idea of autocatalytic sets, and has a certain mathematical naturalness to it.
%

\paragraph{Acknowledgements.}
We thank Jaikumar Radhakrishnan for pointing out that our proof of Theorem~\ref{rank-nullity} showed a result mildly stronger ($\R^n_{<0} \subseteq \op{cone}(\op{rows}(A))$) than what we had asserted in a previous draft ($\op{cone}(\op{rows}(A)) \cap \R^n_{<0}$ is nonempty). We thank the referees for useful comments.

\bibliographystyle{amsplain}

\providecommand{\bysame}{\leavevmode\hbox to3em{\hrulefill}\thinspace}
\providecommand{\R}{\relax\ifhmode\unskip\space\fi MR }

\end{document}